\documentclass[12pt]{amsart}

\usepackage{amsmath}
\usepackage{amsrefs}
\usepackage{hyperref}
\usepackage{amssymb}
\usepackage{amsthm}
\usepackage{bbm}
\usepackage{bm}
\usepackage{dsfont}
\usepackage[margin=1in]{geometry}
\usepackage{paralist}

\newtheorem{lemma}{Lemma}
\newtheorem*{lemma*}{Lemma}
\newtheorem{theorem}{Theorem}

\newtheorem{remark}{Remark}

\theoremstyle{definition}
\newtheorem{definition}{Definition}
\newtheorem{assumption}{Notation}

\def\*{\times}
\def\1{\mathbbm{1}}
\def\a{\mathfrak{a}}
\def\A{\mathbb{A}}

\def\D{\mathcal D}

\def\GL{\text{GL}_2}
\def\H{\mathcal H}
\def\Hp{\mathcal H_{+}}
\def\Hm{\mathcal H_{-}}
\def\He{\mathcal H_{=}}
\def\I{\mathcal I}

\def\L{\mathcal L}

\def\o{\mathfrak{o}}
\def\p{\mathfrak{p}}
\def\P{\mathcal P}
\def\Q{\mathbb{Q}}
\def\R{\mathbb{R}}
\def\SL{\text{SL}_2}
\def\SO{\text{SO}_2}

\def\Z{\mathbb{Z}}

\numberwithin{equation}{section}

\newcommand{\sgn}{\text{sgn}}
\newcommand{\Ad}{\text{Ad}}\def\R{\mathbb{R}}

\title[Sup norms]{Optimal sup norm bounds for newforms on $\GL$ with maximally ramified central character}
\author{F\'{e}licien Comtat}

\begin{document}

\begin{abstract}
Recently, the problem of bounding the sup norms of $L^2$-normalized cuspidal automorphic newforms
$\phi$ on $\GL$ in the level aspect has received much attention. 
However at the moment strong upper bounds are only available  if the central character $\chi$ of $\phi$ is not 
too highly ramified. 
In this paper, we establish a uniform upper bound in the level aspect for general $\chi$.
If the level $N$ is a square, our result reduces to $$\|\phi\|_\infty \ll N^{\frac14+\epsilon},$$
at least under the Ramanujan Conjecture. 
In particular, when $\chi$ has conductor $N$, this improves upon the previous best known bound 
$\|\phi\|_\infty \ll N^{\frac12+\epsilon}$ in this setup (due to Saha~\cite{hybrid})
and matches a lower bound due to Templier~\cite{Templier}, thus our result is essentially optimal in
this case.

\end{abstract}

\maketitle

\section{Introduction}

Let $\phi$ be a cuspidal automorphic form on $\GL(\A_\Q)$ with conductor $N=\prod_p p^{n_p}$ and central
character $\chi$.
Assume in addition $\phi$ is a newform, in the sense that there exists either a Maa{\ss} or
holomorphic
cuspidal newform $f$ of weight $k$ for $\Gamma_1(N)$ such that for all $g \in \SL(\R)$ we have $\phi(g)=j(g,i)^{-k}f(g \cdot i)$,
{  where as usual $j(g,z)=cz+d$ for $g=\begin{bmatrix}
	a      & b \\
	c     & d 
	\end{bmatrix} \in \SL(\R)$ and $z \in \mathbb{H}$}.
In particular, $\phi$ is bounded and {  has a {  finite} $L^2$ norm}, hence one may be interested
in asking how its $L^{\infty}$ and its $L^2$ norm relate. In the level aspect, one traditionally
asks for bounds for $\|\phi\|_\infty=\sup_g|\phi(g)|=\sup_z|y^{\frac{k}2}f(z)|$ depending on $N$ as $\|\phi\|_2$ is fixed. 
Subsequent investigations have shown that it is relevant for this problem to also take into account 
the conductor $C=\prod_p p^{c_p}$ of $\chi$.
 Assuming that $\phi$ is $L^2$-normalized, the ``trivial bound" is
\begin{equation}\label{thetrivialbound}
1 \ll \|\phi \|_{\infty} \ll N^{\frac12+\epsilon}
\end{equation}
for any $\epsilon > 0$. Here and below, the implied constant may depend on $\epsilon$ and on the
archimedean parameters of $\phi$. The upper bound in~(\ref{thetrivialbound}) does not appear to have
been written down previously for general $N$ and $C$, but it can be deduced from the main result
of~\cite{hybrid} for instance.

For squarefree $N$, the first non-trivial upper bound is due to Blomer and Holowinsky~\cite{BH}, and
has been subject to several improvements by Harcos and Templier (and some unpublished work of 
Helfgott and Ricotta) culminating with the result of~\cite{HT} which achieve the upper bound
$N^{\frac13+\epsilon}$. 
For non-squarefree $N$, the best result to date is due
to Saha~\cite{hybrid}, but it significantly improves on the trivial bound only when $\chi$ is not
highly ramified (here and elsewhere we say $\chi$ is highly ramified if 
$c_p > \lceil \frac{n_p}2 \rceil$ for some 
prime $p$). Indeed, if $\chi$ is not highly ramified and $N$ is a perfect square, then Saha's
result~\cite{hybrid} gives an upper bound of $N^{\frac14+\epsilon}$. Recent work of Hu and Saha (see~\cite{HuSaha}, 
especially the last paragraph of their introduction) suggests that this bound may be further improved in the compact case.
On the other hand, if $N=C$ and if $N$ is a perfect square, then Saha's result~\cite{hybrid} 
reduces to the trivial bound~(\ref{thetrivialbound}).

Templier was the first to provide evidence that the actual size of 
$\|\phi\|_\infty$ may depend on how ramified $\chi$ is. Namely, he proved in~\cite{Templier} that
whenever $N=C$ we have 
\begin{equation}\label{lower}
\|\phi\|_\infty \gg N^{-\epsilon}\prod_{p^{n_p} \| N} p^{\frac12 \lfloor \frac{n_p}2 \rfloor}.
\end{equation}
 In particular, if $N$ is a square, then 
\begin{equation*}
\|\phi\|_\infty \gg N^{\frac14-\epsilon}.
\end{equation*}
We shall prove the following comparable upper bound, which improves on~\cite{hybrid} when $\chi$
is highly ramified. 

\begin{theorem}\label{thmgen}
Let $\pi$ be an unitary cuspidal automorphic representation of 
$\text{GL}_2(\A_{\Q})$ with central character $\omega_\pi$. 
Let $N=\prod_p p^{n_p}$ be the conductor of $\pi$.
Let $\phi \in \pi$ be an $L^2$-normalized  newform.
Then 
$$\|\phi\|_{\infty} \ll_{\epsilon, \pi_\infty} N^{\delta+\epsilon} 
\prod_{p|N}p^{\frac12\lceil \frac{n_p}2 \rceil},$$
where $\delta$ is any {  bound} towards  {  the} Ramanujan Conjecture for $\pi$.
\end{theorem}

Theorem~\ref{thmgen} provides for the first time non-trivial upper bounds for general $N$
that do not get worse when the conductor $C$ varies. As a point of comparison, the main result
of~\cite{hybrid} had an additional factor of $\prod_{p} p^{\max\{0,c_p-\lceil \frac{n_p}2 \rceil\}}$,
which is larger than one precisely when $\chi$ is highly ramified.
Furthermore, for $C=N$, in view of the lower bound~(\ref{lower}) and assuming {  the} Ramanujan Conjecture, 
our result is essentially optimal when $N$ is a square.
Note that  {  the} Ramanujan Conjecture is known by work of Deligne and Serre for $\phi$ arising from a 
holomorphic cusp form, and otherwise $\delta=\frac7{64}$ is admissible~\cite{KimSarnak}.

\begin{remark}
In~\cite{hybrid}, the appeal to {  a bound} towards {  the} Ramanujan Conjecture is avoided by using
H\"older inequality to estimate separately $L^2$ averages of the Whittaker newforms at primes 
{  at which the central character is ramified}
and moments of the coefficients $\lambda_\pi$ of the L-function attached to $\pi$. 
However, in our situation, we want to exploit the fact that the Whittaker coefficients are supported
on arithmetic progressions of modulus $L$, say, as explained later. 
A similar technique as in~\cite{hybrid} would thus lead us to estimate moments of $\lambda_\pi$ on
these arithmetic progressions. 
{  One might expect that these moments are approximately $L$ times smaller than the full moments,
but such a result does not seem to be available. Hence,}
if we were to bound them by positivity by the full moments, we would expect an over-estimate {  of same order as $L$.}
Since estimates are known by Rankin-Selberg theory up to the eighth moments, and, as we shall see,
$L \le \prod_{c_p > \frac{n_p}2}p^{\lfloor \frac{n_p}2 \rfloor}$, one should be able to replace 
$N^\delta$ in Theorem~\ref{thmgen} with 
$\prod_{c_p > \frac{n_p}2}p^{\frac18\lfloor \frac{n_p}2 \rfloor }$, similarly as in Theorem~1.1 of
\cite{minimaltype}. 
However, for the sake of brevity, we do not carry out this argument.
\end{remark}

The lower bound~(\ref{lower}) has been generalized by
Saha in \cite{largevalues} and subsequently by Assing in~\cite{TAMS}.
When $\chi$ is not maximally ramified, there is still a gap between the best known lower bound and the
upper bound from Theorem~\ref{thmgen}.
Finally, let us mention that the hybrid bounds over $\Q$ in~\cite{hybrid}, which combines the Whittaker expansion with some amplification, still beats our result
when $\chi$ is not highly ramified. 
For hybrid bounds over general number fields, we refer to the work of Assing~\cites{Assing,thesis}.

The proof proceeds by using Whittaker expansion to reduce the problem of bounding $\phi$ to
that of understanding the local newforms attached to $\phi$. By making use of the
invariances of $\phi$, we can restrict ourselves to evaluate these local newforms in the Whittaker model
on some convenient cosets. 
{  The values of these local newforms have been computed~\cites{TAMS, thesis}} by using a ``basic identity" {  derived} 
from the Jacquet-Langlands local functional equations which was first {  expressed in this form} in~\cite{largevalues}. 
In the non maximally ramified case, local bounds are slightly weaker than needed to obtain our result, and 
we take advantage of strong $L^2$-bounds due to Saha~\cite{hybrid} instead.

Actually, we are using the Whittaker expansion of a certain translate of $\phi$, 
the ``balanced newform". The main feature is that it is supported on arithmetic progressions, which
enables us to get some savings. 
Though we are working adelically, this fact can also be seen classically by computing the Fourier
expansion of the corresponding cusp form at cusps of large width.
The situation is somewhat analogous to \cite{minimaltype}, where the authors also get Whittaker 
expansions supported on arithmetic progressions. 

Let us explain this analogy in the maximally ramified case -- in which we get optimal upper bounds.
As we shall see, in this case each local representation {  with ramified central character} is of the form  
$\chi_1 \boxplus \chi_2$, where $\chi_1$ has exponent of conductor $n_p$ and $\chi_2$ is unramified.
Then the local balanced newform for $\pi$ is a twist of the local balanced newform for 
$\chi_1 \chi_2^{-1}\boxplus 1$. 
For representations of this type, the local balanced newform coincides with the $p$-adic microlocal
lift as defined in~\cite{Nelson}. 
Now as explained in~\cite{minimaltype}, the microlocal lift is the split analogue
of the minimal vectors used there.
Therefore the fact that we get optimal sup norm bounds in this case is the direct analogue of
Theorem~1.1 of~\cite{minimaltype} which gives an optimal sup norm bound for automorphic forms of minimal
type.

It is worth noticing that~\cite{minimaltype},~\cite{compact} as well as the present work provide
instances of the seemingly general principle according to which when considering very localized
vectors, one is able to establish very good and sometimes optimal upper bounds.
This is even the case when a Whittaker expansion is not available, as in~\cite{compact}.

The analysis of local newforms is given in Section~\ref{local}. 
The proof of Theorem~\ref{thmgen} is given in Section~\ref{global}. 

\subsection*{Acknowledgement} I wish to thank Abhishek Saha for suggesting me this problem as well as helpful comments and discussions,
and the anonymous referee whose suggestions helped improving this paper.

\section{Local bounds}\label{local}
In this section, $F$ will denote a non-archimedean local field of characteristic zero {  with}
residue field $\mathbb F_q$.
Let $\o$ denote the ring of integers of $F$ and $\p$ its maximal ideal with uniformizer
$t_\p$. The discrete valuation associated to $F$ will be denoted by $v_\p$. We define $U(0)=\o^\*$,
and for $k \ge 1$, $U(k)=1+\p^k$.
We fix an additive unitary character $\psi$ of $F$ with conductor $\o$. 
{  In the sequel, the Whittaker models given will be those with respect to $\psi$.}

\subsection{Generalities}
\subsubsection{Double coset decomposition}
Let $G=\text{GL}_2(F)$, $K=\text{GL}_2(\o)$. 
For $x \in F$ and $y \in F^\*$, consider the following elements
$$
w=\begin{bmatrix}
   0       & 1 \\
   -1      & 0 
\end{bmatrix},
a(y)=\begin{bmatrix}
   y      & 0 \\
    0     & 1 
\end{bmatrix},
n(x)=\begin{bmatrix}
   1       & x \\
   0    &  1 
\end{bmatrix},
z(y) = \begin{bmatrix}
   y       & 0 \\
    0     & y
\end{bmatrix}.
$$
Then define the following subgroups
$$N=n(F), \> A=a(F^\*), \> Z=z(F^\*),$$
and, for $\a$ an ideal of $\o$,
\begin{equation}
\label{K1}
K^{(1)}(\a)=K \cap 
\begin{bmatrix}
   1+\a       & \o \\
    \a      & \o 
\end{bmatrix},
\>
K^{(2)}(\a)=K \cap 
\begin{bmatrix}
   \o     & \o \\
    \a      & 1+\a
\end{bmatrix}.
\end{equation}
Note that  for $\a=\p^n$, with $n$ a non-negative integer, we have
\begin{equation}
\label{conjugateK2}
K^{(2)}(\p^n)=\begin{bmatrix}
       & 1\\
    t_\p^n & 
\end{bmatrix}K^{(1)}(\p^n)
\begin{bmatrix}
       & 1\\
    t_\p^{n} & 
\end{bmatrix}^{-1}.
\end{equation}
From \cite{largevalues}*{Lemma 2.13}, for any integer $n \ge 0$ we have the following double coset decomposition
\begin{equation}\label{dcd}
G=\coprod_{m \in \Z} \coprod_{\ell=0}^n \coprod_{\nu \in \o^\* /(1+\p^{\ell_n})}
ZNg_{m,\ell,\nu}K^{(1)}(\p^n),
\end{equation}
where $\ell_n=\min \{\ell, n-\ell\}$, and 
\begin{align*}
g_{m,\ell,\nu}&=a(t_\p^m)wn(t_\p^{-\ell}\nu)\\
&=\begin{bmatrix}
   0       & t_\p^m \\
  -1    & -t_\p^{-\ell}\nu
\end{bmatrix}.
\end{align*}

\begin{definition}\label{gmellnu}
Assume $n \ge 0$ is a fixed integer. Then for any $g \in G$ we define { 
$$(m(g),\ell(g),\nu(g)) \in \Z \times \{0 , \cdots , n\} \times \o^\* /(1+\p^{\ell(g)_n})$$}
as the unique triple 
such that $$g \in ZNg_{m(g),\ell(g),\nu(g)}K^{(1)}(\p^n).$$
\end{definition}

\begin{remark}\label{coset}
Any $g \in \GL(F)$ belongs to some $ZNa(y)\kappa$ where 
$\kappa=\begin{bmatrix}
   a      & b \\
    c      & d 
\end{bmatrix} \in \GL(\o)$. Then by Remark~2.1 of \cite{hybrid}, 
we have $\ell(g)=\min\{v_\p(c),n\}$ and $m(g)=v_\p(y)-2\ell(g)$.
In particular, if $g$ is already an element of $\GL(\o)$, then $g$ is in a
coset of the form $g_{-2j,j,*}$
\end{remark}

Now we determine the double cosets corresponding to certain elements of interest for the global application.

\begin{lemma}\label{glotolo}
Consider two integers $0 \le e \le n$.
Let $g \in \GL(\o)a(t_\p^e)$.
Then there exist a non-negative integer $\ell \le n$ and $\nu \in \o^\*$ such that one of the following holds
\begin{enumerate}
\item  either $\ell \le e$ and $g \in ZNg_{-e,\ell,\nu}K^{(1)}(\p^{n}),$
\item or $e < \ell \le n$ and $g \in ZNg_{-2\ell+e,\ell,\nu}K^{(1)}(\p^{n}),$
\end{enumerate}
where the subgroup $K^{(1)}(\p^{n})$ is defined in~(\ref{K1}).
\end{lemma}
\begin{proof}
We know by~(\ref{dcd}) that $g \in ZNg_{m,\ell,\nu}k_1$ for some $k_1 \in K^{(1)}(\p^{n})$ hence 
$$g k_1^{-1}a(t_\p^{-e})
\in ZNg_{m,\ell,\nu}a(t_\p^{-e}).$$
Since $g \in Ka(t_\p^{e})$, it follows that $gk_1^{-1}a(t_\p^{-e}) \in K$.
By Remark~\ref{coset}, 
it is then in the coset of some $g_{-2j,j,*}$ with $0 \le j \le n$. 
On the other hand,
\begin{align*}
g_{m,\ell,\nu}a(t_\p^{-e}) &= a(t_\p^{m})wn(t_\p^{-\ell}\nu)a(t_\p^{-e})\\
&=a(t_\p^{m})wa(t_\p^{-e})n(t_\p^{e-\ell}\nu)\\
&=t_\p^{-e}a(t_\p^{m+e})wn(t_\p^{e-\ell}\nu).
\end{align*}
If $\ell \le e$ then 
$$wn(t_\p^{e-\ell}\nu) =\begin{bmatrix}
    &  1\\
   -1  & -t_\p^{e-\ell}\nu\\
\end{bmatrix} \in \GL(\o)$$ so by
Remark~\ref{coset} {  $a(t_\p^{m+e})wn(t_\p^{e-\ell}\nu)$} is in the coset of
$g_{m+e,0,*}$. So in this case, $g_{-2j,j,*}=g_{m+e,0,*}$ thus
$m=-e$ and we find that 
$$g \in ZNg_{-e,\ell,\nu}K^{(1)}(\p^{n}).$$
Otherwise $a(t_\p^{m+e})wn(t_\p^{e-\ell}\nu)=g_{m+e,\ell-e,\nu}$, 
therefore $g_{-2j,j,*}=g_{m+e,\ell-e,\nu}$ and we get 
$m+e=-2(\ell-e)$, so
$$g \in ZNg_{-2\ell+e,\ell,\nu}K^{(1)}(\p^{n}).$$
\end{proof}

\subsubsection{Characters and representations}
 {  For $\chi$ a character of $F^\*$, we denote by $a(\chi)$ the exponent of the conductor of
$\chi$, that is the least non-negative integer $n$ such that $\chi$ is trivial on $U(n)$.} 
For $\pi$ an irreducible admissible representation of $G$, we also denote by $a(\pi)$ the
exponent of the conductor of $\pi$, that is the least non-negative integer $n$ such that 
$\pi$ has a $K^{(1)}(\p^n)$-fixed vector. The central character of $\pi$ will be denoted by
$\omega_\pi$.

\subsubsection{The local Whittaker newform}

Fix $\pi$ a generic irreducible admissible unitarizable representation of $G$.
From now on, we fix $n=a(\pi)$, and we shall assume {  that} $\pi$ is realized on its Whittaker model.

\begin{definition}
The normalized newform $W_\pi$ attached to $\pi$ is the unique 
{  $K^{(1)}(\p^n)$}-fixed vector such that $W_\pi(1)=1$.

The normalized conjugate-newform $W_\pi^*$ attached to $\pi$ is the unique 
{  $K^{(2)}(\p^n)$}-fixed vector such that $W_\pi^*(1)=1$.
\end{definition}  

\begin{remark}\label{conjnew}
By~(\ref{conjugateK2}), the function
$$g \mapsto W_\pi\left(g\begin{bmatrix}
       & 1\\
    t_\p^{n} & 
\end{bmatrix}\right)$$
is $K^{(2)}(\p^n)$-invariant. 
Thus there exists a complex number $\alpha_\pi$ such that 
$$W_\pi\left( \cdot \begin{bmatrix}
       & 1\\
    t_\p^{n} & 
\end{bmatrix}\right)= \alpha_\pi W_\pi^*.$$
In addition, we have $W_\pi^*(g)=\omega_\pi(\det(g))W_{\tilde{\pi}}(g)$, where $\tilde{\pi}$ is the
contragradient representation  to $\pi$. Altogether, we get that
$$W_\pi\left( \cdot \begin{bmatrix}
       & 1\\
    t_\p^{n} & 
\end{bmatrix}\right)=\alpha_\pi \omega_\pi(\det(g))W_{\tilde{\pi}}(g).$$
One can even show that $|\alpha_\pi|=1$ (see~\cite{largevalues}*{Lemma~2.17}, or~\cite{largevalues}*{Proposition~2.28}
for an exact formula in terms of $\epsilon$-factors).
Also note the following identity
{ 
\begin{equation}\label{mirror}
n(t_\p^{\ell+m}\nu^{-1})z(t_\p^{\ell-n}\nu^{-1})
g_{m,\ell,\nu}
\begin{bmatrix}
       & 1\\
    t_\p^{n} & 
\end{bmatrix}
=
g_{m+2\ell-n,n-\ell,-\nu}\begin{bmatrix}
      1 & \\
     & -\nu^{-2}
\end{bmatrix},
\end{equation} }
which, combined with the above, enables one to restrict attention to those cosets satisfying $\ell \le \frac{n}2$, 
at the price of changing $\pi$ to $\tilde{\pi}$.
\end{remark}

{  {  Assing has computed the local Whittaker newforms in great generality, and estimated them 
		using the $p$-adic stationary phase method~\cites{TAMS, thesis}.} Let us briefly explain the basic ideas of his method.
For any fixed $m \in \Z$ and $0 \le \ell \le n$ the function on $\o^\*$ given by 
$\nu \mapsto W_\pi(g_{m,\ell,\nu})$ only depends on $\nu \mod (1+\p^\ell)$. Thus, by Fourier inversion, there exist
complex numbers $c_{m,\ell}(\mu)$ such that

\begin{equation*}
W_\pi(g_{m,\ell,\nu})=\sum_{\mu \in \tilde{X}(\ell)}c_{m,\ell}(\mu)\mu(\nu),
\end{equation*}
where $\tilde{X}(\ell)$ is the set of characters $\mu$ satisfying $\mu(t_\p)=1$ and $a(\mu) \le \ell$.

Then, one may reformulate the Jacquet-Langlands local functional equation as an equality of power series 
in the variable $q^s$ whose coefficients involve on one side the Fourier coefficients  $c_{m,\ell}(\mu)$ one is 
interested in, and on the other side Gauss sums and values of the local newform at some diagonal matrices, 
both of which are known~\cite{RS}. This is the content of~\cite{largevalues}*{Proposition 2.23}.
By identifying the coefficients of the power series appearing in both side, one is then able to compute inductively 
the coefficients $c_{m,\ell}(\mu)$, and, from there, the values of the local newform on each double coset.

This can be done for each local representation $\pi$,}
however {  Lemma~\ref{principalseries} below (same as~\cite{largevalues}*{Lemma~2.36})} will enable us to restrict ourselves to principal
series representations. By Remark~\ref{conjnew},
we can further restrict ourselves to the situation $\ell \le \frac{n}2$. 
Finally, as we mentioned earlier, in our global application we shall use Saha's strong $L^2$-bound~\cite{hybrid}, so what we are really
interested in this section is only the support of the local newforms.

\begin{lemma}\label{principalseries}
Assume $a(\omega_\pi)>\frac{a(\pi)}2$.
Then $\pi = \chi_1 \boxplus \chi_2$, where $\chi_1$ and $\chi_2$ are
 unitary characters with respective exponents of conductors $a_1=a(\omega_\pi)$ and $a_2=n-a(\omega_\pi)$.
\end{lemma}

In the rest of this section, we shall only consider the case $a(\omega_\pi)>\frac{a(\pi)}2$,
as the main point of our global application is to take advantage of primes {  at which the central character is
highly ramified}.
Thus for our purpose, we only have to consider $\pi = \chi_1 \boxplus \chi_2$
with $a_2 < \frac{n}2 < a_1$, where from now on we denote $a_1 = a(\chi_1)$ and $a_2=a(\chi_2)$.
We first state the case of maximally ramified principal series.

\begin{lemma}\label{Wvaluesmax}
Let $\pi$ be a generic irreducible admissible unitarizable representation of $G$
with exponent of conductor $a(\pi)=n>1$.
Assume $a(\omega_\pi)=a(\pi)$. Then there exists $\nu_1 \in \o^\*$ such that for 
all $m \in \Z$ and for $0\le \ell \le \frac{n}2$, we have 
\begin{equation*}
|W_\pi(g_{m,0,\nu})|=\1_{m \ge -n}q^{-\frac{m+n}2},
\end{equation*}
\begin{align*}
|W_\pi(g_{-n-\ell,\ell,\nu})|=\left\lbrace
\begin{array}{cc}
q^{\frac{\ell}2}  & \mbox{if } \nu \in \nu_1 +\p^\ell,\\
0 & \mbox{if } \nu \not \in \nu_1 +\p^\ell,\\ 
\end{array}\right.
\end{align*}
and if $0 < \ell < n$ and $m+\ell \neq -n$ then $W_\pi(g_{m,\ell,\nu})=0$.
\end{lemma}
\begin{proof}
This follows from Lemma~3.4 and proof of Lemma~5.8 in~\cite{TAMS}.
\end{proof}
In particular, one sees that in this case the local Whittaker newform is essentially supported on {  an} 
arithmetic progressions.
The case $1\le a_2 < \frac{n}{2} < a_1$ is a bit more complicated, but one may obtain a result similar in flavour.
{  Work of Assing~\cite{thesis} gives precise bounds for the local newform, however these local bounds are slightly
	weaker than what we need for our global application. Consequently, we only give here statements regarding the support
	of the local newform, and we shall rely on strong bounds for the $L^2$ mass~\cite{hybrid}.}

\begin{lemma}\label{Wvalueshigh}
Let $\pi$ be a generic irreducible admissible unitarizable representation of $G$
with exponent of conductor $n>1$.
Assume $\frac{n}2<a(\omega_\pi)<n$. Set $a_1=a(\omega_\pi)$ and $a_2=n-a_1$. 
Assume moreover $F=\Q_p$.
{  There exists $\nu_1 \in \o^\*$ such that if $m \in \Z$ and $0 \le \ell \le \frac{n}2$, 
	then have $W_\pi(g_{m,\ell,\nu})=0$ unless one of the following holds:
\begin{enumerate}
	\item $\ell < a_2$ and $m = -n$,
	\item $\ell = a_2$ and $m \ge -n$,
	\item $\ell > a_2$, $m=-a_1-\ell$ and $\nu \in {\nu_1}^{-1} + {t_\p}^{\ell -a_2} \o^\*$
\end{enumerate}
}
\end{lemma}	{ 
\begin{proof} 
 This follows almost directly from inspection of the cases in Lemma~3.4.12 in~\cite{thesis}.
	Since we are taking $F=\Q_p$, the quantity $\kappa_F$ defined in~\cite{thesis} equals one, so the only 
	bothersome case is $a_2 < \ell \le \frac{a_1+a_2}{2}$ when $a_2=1$.
	By~\cite{thesis}*{Lemma~3.3.9}, for $a_2 < \ell < a_1$ we must have $m=-a_1-\ell$, so it only remains to 
	see that the congruence condition also holds.
	If $\ell \le \frac{a_1}2,$ this follows from Case~I of the proof of Lemma~\cite{thesis}*{Lemma~3.4.12}.
	The only remaining case is thus $\ell=\frac{1+a_1}2$, which only occurs for $a_1$ odd, hence $a_1 \ge 3$, 
	so $a_1-a_2 \ge 2\kappa_F$. 
	As seen from Case~VI.2 of the proof, this last condition is enough to get the congruence condition.
\end{proof}}

\subsection{Archimedean case}
The local representation at the infinite place is a generic irreducible admissible unitary
representation $\pi$ of $\text{GL}_2(\R)$. 
Let $\psi$ be the additive character of $\R$ given by $\psi(x)=e^{2i\pi x}$.
The lowest weight vector in the Whittaker model with respect to $\psi$ 
is given by 
\begin{equation}\label{lowestwv}
W_\pi(n(x)a(y))=e^{2i\pi x}\kappa(y),
\end{equation} where
$\kappa$ is determined by the form of the representation $\pi$. We shall use that for $y \in \R^\*$
\begin{equation}\label{estimate}
\kappa(y)\ll |y|^{-\epsilon}e^{(-2\pi+\epsilon)|y|}.
\end{equation}
uniformely in $y$. To see this, let us examine the possibilities for $\pi$.
\subsubsection{Principal series representations}
If $\pi=\chi_1 \boxplus \chi_2$, where $\chi_i=\sgn^{m_i}|.|^{s_i}$ with $0 \le m_2 \le m_1 \le 1$
integers and $s_1+s_2 \in i\R$ and $s_1-s_2 \in i\R \cup (-1,1)$ then the lowest weight vector is
given by
$$\kappa(y)=\left\lbrace\begin{array}{cccc}
\sgn(y)^{m_1}|y|^{\frac{s_1+s_2}2}|y|^{\frac12}K_{\frac{s_1-s_2}2}(2\pi|y|)& \mbox{if}&m_1=m_2\\
|y|^{\frac{s_1+s_2}2}|y|\left(K_{\frac{s_1-s_2-1}2}(2\pi|y|)+\sgn(y)K_{\frac{s_1-s_2+1}2}(2\pi|y|)\right)& \mbox{if}&m_1 \neq m_2,
\end{array}
\right.$$
where $K_\nu$ is the $K$-Bessel function of index $\nu$. By \cite{HM}*{Proposition~7.2}, we
have the following estimate.
\begin{lemma}
Let $\sigma>0$. For $\Re(\nu)\in (-\sigma,\sigma)$ we have
$$
K_\nu(u) \ll_\nu \left\lbrace\begin{array}{ccc}
u^{-\sigma-\epsilon} & \mbox{if} & 0<u\le 1+\frac{\pi}2\Im(\nu),\\
u^{-\frac12} e^{-u} & \mbox{if} & u >  1+\frac{\pi}2\Im(\nu).
\end{array}
\right.
$$
\end{lemma}
In particular, taking $\sigma=\frac12$ if $m_1=m_2$ and $\sigma=1$ otherwise,~(\ref{estimate}) follows in this case.
\subsubsection{Discrete series representations}
If $\pi$ is the unique irreducible subrepresentation of $\chi_1 \boxplus \chi_2$, where
 $\chi_i=\sgn^{m_i}|.|^{s_i}$ with $0 \le m_2 \le m_1 \le 1$
integers and $s_1+s_2 \in i\R$ and $s_1-s_2 \in \Z_{>0}$,  $s_1-s_2 \equiv m_1-m_2+1 \mod 2$, 
then the lowest weight vector is given by
$$\kappa(y)=|y|^{\frac{s_1+s_2}2}y^{\frac{s_1-s_2+1}2}(1+\sgn(y))e^{-2\pi y},$$
and we see that it satisfies again the estimate~(\ref{estimate}).

\section{Global computations}\label{global}
\subsection{Notations}\label{notation}
Let $\A_{\Q}$ denote the ring of ad\`eles of $\Q$ and let $\psi$ be the unique additive character 
of $\A_\Q$ that is unramified at each finite place and equals $x \mapsto e^{2i\pi x}$ at $\R$. 
 For any local object defined in Section~\ref{local}, we use the subscript $_p$ to denote this object defined over $\Q_p$.
 We also fix in all the sequel 
\begin{equation}\label{so2}
\Gamma_\infty = \SO(\R).
\end{equation}
Let $\pi=\otimes_{p \le \infty} \pi_p$ be a unitary cuspidal automorphic representation of 
$\text{GL}_2(\A_{\Q})$ with central character $\omega_\pi$. Let $N=\prod_pp^{n_p}$ be the conductor
of $\pi$ and let $C=\prod_p p^{c_p}$ be the conductor of $\omega_\pi$. In particular $C \mid N$.
Let us introduce some notation to denote respectively the set of primes for which
Lemma~\ref{principalseries} do or do not apply, namely
\begin{equation}
\label{highlyramifiedprimes}
\H=\left\{p \mid N : c_p > \frac{n_p}2\right\}
 \text{ and }
\L=\left\{p \mid N : c_p \le \frac{n_p}2\right\}.
\end{equation}
We also denote by $S_N$ the set of prime numbers dividing $N$, so that
$$S_N= \H \cup \L.$$
 Then according to Lemma~\ref{principalseries}, $\pi_p$ is an irreducible principal series
 representation for each prime $p \in \H$ , and we have corresponding local exponents of conductors
 $a_1(p)=c_p$ and $a_2(p)=n_p-c_p$.
 Finally, for any set of primes $\P$, define $\Psi(\P)$ to be
the set of positive integers having all their prime divisors among $\P$. We shall use the following obvious result.

\begin{lemma}\label{smoothpower}
Let $\P$ be a finite set of primes.
Then for all $0<\alpha \le \frac1{\log(2)}$ we have 
$$\sum_{s \in \Psi(\P)}s^{-\alpha} = \prod_{p \in \P}\frac{1}{1-p^{-\alpha}} \le \left(\frac{2}{\alpha \log 2} \right) ^{\# \P}$$
\end{lemma}

\subsection{The Whittaker expansion}
Let $\phi \in \pi$ be an $L^2$-normalized  newform.
Define the global Whittaker newform on $\text{GL}_2(\A_{\Q})$ by
$$W_\phi(g)=\int_{\Q \backslash \A_{\Q}}\phi(n(x)g)\psi(-x)dx.$$ 
It factors as
$$W_\phi(g)=c_\phi\prod_{p \le \infty}W_p(g_p),$$
where $W_p$ are as defined in the first two sections, and $c_\phi$ is a constant that
satisfies $$2\xi(2)c_{\phi}^2\|\prod_{p \le \infty}W_p\|^2_{reg}=1,$$
with $$\|\prod_{p \le \infty}W_p\|_{reg}=L(\pi,\Ad,1)\prod_{p\le \infty}\frac{\zeta_p(2)\|W_p\|_2}{\zeta_p(1)L_p(\pi,\Ad,1)},$$
see \cite{subconvex}*{Lemma~2.2.3}. 
In turn, we have the Whittaker expansion
\begin{equation}
\label{WE}
\phi(g)=\sum_{q \in \Q^\*}W_{\phi}(a(q)g)=c_\phi\sum_{q \in \Q^\*}\prod_{p \le \infty}W_p(a(q)g_p)
\end{equation}
for any $g \in \text{GL}_2(\A_\Q)$. 
Our strategy to bound $\|\phi\|_\infty$ will be to bound for all $g$
$$
|c_\phi|\sum_{q \in \Q^\*}\prod_{p \le \infty}|W_p(a(q)g_p)| \ge |\phi(g)|,
$$
that is, we do not take advantage of the potential oscillations in the Whittaker expansion.
First, we give a bound for the constant
$c_\phi$ appearing here. By \cite{JHPL} we have 
$$L(\pi,Ad,1) \gg N^{-\epsilon}.$$
For $p$ unramified, 
$$\frac{\zeta_p(2)\|W_p\|_2}{\zeta_p(1)L_p(\pi,\Ad,1)}=1.$$
For $p$ ramified, we have
$$L_p(\pi,\Ad,1) \asymp 1 \mbox{ and } 1 \le \|W_p\|_2 \le 2$$
(see \cite{largevalues}*{Lemma~2.16}). Consequently, $|c_\phi| \ll N^\epsilon$. We shall also use that
for any integer $n$ coprime to $N$, we have 
\begin{equation}\label{Lfunction}
\prod_{p \nmid N}W_p(a(n))=n^{-\frac12}\lambda_{\pi}(n),
\end{equation}
where $\lambda_\pi(n)$ is the $n$-th coefficient of the finite part of the $L$-function attached 
to $\pi$.

\subsection{Generating domains}
Using invariances of automorphic forms, we can restrict their argument to lie in some
convenient set of representatives. We first describe such generating domains.

\begin{definition}\label{DN}
We denote by $\D_N$ be the set of $g \in \GL(\A_\Q)$ such that
\begin{itemize}
\item $g_\infty=n(x)a(y)$ for some $x \in \R$ and $y \ge \frac{\sqrt 3}2$, 
\item $g_p=1$ for all $p \nmid N$,
\item $g_p \in \GL(\Z_p)$ for all $p$.
\end{itemize}
\end{definition}

\begin{lemma}\label{gegendom}
Let $\Gamma=\prod_{p \le \infty} \Gamma_p$ be a subgroup of $\GL(\A_\Q)$ such that $\Gamma_\infty=\SO(\R)$,  
for all finite $p$ the group $\Gamma_p$ is an open
subgroup of $\GL(\Z_p)$ whose image by the determinant map is $\Z_p^\*$, and 
$\Gamma_p=\GL(\Z_p)$ for $p \nmid N$.
Then the subset $\D_N$ of $\GL(\A_\Q)$ given by Definition~\ref{DN} contains representatives of each double coset of 
$Z(\A_\Q)\GL(\Q) \backslash \GL(\A_\Q) / \Gamma$.
\end{lemma}
\begin{proof}
By the strong approximation theorem,
any $g \in \GL(\A_\Q)$ can be written as $g_\infty\gamma k$ with 
$g_\infty\in \GL^+(\R)$, $\gamma \in \GL(\Q)$, and $k \in \Gamma$.
Multiplying on the left by $\gamma^{-1}$ and on the right by $k^{-1}$, we 
can first assume that $g_p=1$ for all finite $p$. Next, let $z=g_\infty \cdot i$.
Then there is $\sigma \in \SL(\Z)$ such that 
$\Im(\sigma \cdot z) \ge \frac{\sqrt{3}}2$.
After multiplying on the left by $\sigma$ and on the right by 
$\prod_{p \nmid N} \sigma^{-1}$, we can instead assume that $g_p=1$ for all 
$p \nmid N$, $g_p \in \GL(\Z_p)$ for $p | N$, and 
$\Im(g_\infty z) \ge \frac{\sqrt{3}}2$.
Finally, multiplying by an element of $\SO(\R)$, we can assume that $g_\infty$
is of the form $n(x)a(y)$ with $y \ge \frac{\sqrt{3}}2$.
\end{proof}

Instead of evaluating our newform $\phi$ on elements of our generating domain $\D_N$, 
we shall rather use it with a certain translate of $\phi$, the ``balanced newform".

\begin{lemma}\label{gendom}
Consider the subgroup $K^{(1)}$ of $\GL(\A_Q)$ defined by $K^{(1)}=\Gamma_\infty \prod_{p < \infty} K^{(1)}(p^{n_p}\Z_p)$,
where the local subgroups $\Gamma_\infty$ and $K^{(1)}(p^{n_p}\Z_p)$ are defined in~(\ref{so2}) and~(\ref{K1}) respectively.
For each prime $p$ dividing $N$, let $e_p$ be an integer with $0 \le e_p \le n_p$.
Let $\D_N$ be the subset of $\GL(\A_\Q)$ given by Definition~\ref{DN}.
Then the set $\D_N \prod_{p \mid N} a(p^{e_p}) \subset \GL(\A_\Q)$ contains representatives of each double coset of 
$Z(\A_\Q)\GL(\Q) \backslash \GL(\A_\Q) / K^{(1)}$.
\end{lemma}
\begin{proof}
Let $$\Gamma_p = \GL(\Z_p) \cap \begin{bmatrix}
   1+p^{n_p}\Z_p      & p^{e_p}\Z_p \\
    p^{n_p-e_p}\Z_p      & \Z_p 
\end{bmatrix},
$$
and $\Gamma=\prod_{p \le \infty} \Gamma_p$.
Let $g \in \GL(\A_\Q)$. By Lemma~\ref{gegendom} there exists $g_d \in \D_N$ such that we have the following equality of
double cosets
$$Z(\A) \GL(\Q) g \prod_p a(p^{-e_p}) \Gamma = Z(\A) \GL(\Q) g_d \Gamma.$$ 
In particular, for each $p \mid N$ there exists $k_p \in \Gamma_p$ such that 
$$Z(\A) \GL(\Q) g \prod_p a(p^{-e_p}) = Z(\A) \GL(\Q) g_d k_p.$$
Now if 
 $$
k_p=
\begin{bmatrix}
   1+ap^{n_p}      & bp^{e_p} \\
    cp^{n_p-e_p}      & d 
\end{bmatrix},
$$
then 
$$a(p^{-e_p})k_pa(p^{e_p})=\begin{bmatrix}
   1+ap^{n_p}       & b \\
    cp^{n_p}      & d 
\end{bmatrix} \in K^{(1)}(p^{n_p}\Z_p).$$
Hence writing 
\begin{align*}
Z(\A) \GL(\Q)g&=Z(\A) \GL(\Q)g_d\prod_{p \mid N}k_pa(p^{e_p})\\
&=Z(\A) \GL(\Q)(g_d\prod_{p \mid N}a(p^{e_p}))\prod_{p \mid N}a(p^{-e_p})k_pa(p^{e_p}),
\end{align*}
we find that the double coset $Z(\A) \GL(\Q)gK^{(1)}$ contains the element 
$g_d\prod_{p \mid N}a(p^{e_p}) \in \D_N \prod_{p \mid N} a(p^{e_p})$.
\end{proof}

By Lemma~\ref{gendom}, we can restrict ourselves to evaluate $|\phi|$ on 
$\D_N \prod_pa(p^{e_p})$, where the exponents $e_p$ may be conveniently chosen. 
Of course, this is equivalent to evaluate its right translate by 
$\prod_pa(p^{e_p})$ on $\D_N$.
Now, by Lemma~\ref{glotolo} of Section~\ref{local}, we can describe this generating domain in terms of
the explicit representatives corresponding to each local double
coset decomposition.

\begin{lemma}\label{L1+DN}
Let $\D_N$ be the subset of $\GL(\A_\Q)$ given by Definition~\ref{DN}.
Let $g \in \D_N \prod_{p \mid N} a(p^{e_p}) \subset \GL(\A_\Q)$. Then $g$ satisfies the following.
\begin{itemize}
\item $g_\infty=n(x)a(y)$ for some $x \in \R$ and $y \ge \frac{\sqrt 3}2$, 
\item $g_p=1$ for all $p \nmid N$,
\item  Let $p \mid N$. If $\ell(g_p) \le e_p$ then $m(g_p)=-e_p$, and if $\ell(g_p) > e_p$ then $m(g_p)=-2\ell(g_p)+e_p$,
where we have used notations of Definition~\ref{gmellnu}.
\end{itemize}
\end{lemma}
\begin{proof}
This follows immediately from Definition~\ref{DN} and Lemma~\ref{glotolo}.
\end{proof}

In particular the (optimal) choice $e_p= \lfloor \frac{n_p}2 \rfloor$ for all $p \mid N$, 
together with Remark~\ref{conjnew}, motivates the following definition.

\begin{definition}\label{IN}
Let $\I_N$ be the set of $g \in \GL(\A_\Q)$ such that
\begin{itemize}
\item $g_\infty=n(x)a(y)$ for some $x \in \R$ and $y \ge \frac{\sqrt 3}2$, 
\item $g_p=1$ for all $p \nmid N$,
\item for all $p \mid N$ we have $\ell(g_p) \le \frac{n_p}2$ and 
$m(g_p) \in \left\{ -\lfloor \frac{n_p}2 \rfloor,  -\lceil \frac{n_p}2 \rceil \right\}$.
\end{itemize}
\end{definition}
{  
\begin{remark}
	Note that for $p \mid N$ we do not require $g_p \in \GL(\Z_p)a(p^{e_p})$, but only the stated 
	conditions about $\ell(g_p)$ and $m(g_p)$.
\end{remark}}

Finally, let us state the quantity we shall actually bound.

\begin{lemma}\label{locgendom}
Recall notations from \S~\ref{notation}. For each $S \subset S_N$, define 
$$\phi^S(g)=\phi\left(g\prod_{p \in S}\begin{bmatrix}
       & 1\\
    p^{n_p} & 
\end{bmatrix}\right).$$
Then 
\begin{equation}\label{supofphi}
\|\phi\|_{\infty} = \max_{S \subset S_N}\sup_{g \in \I_N} | \phi^S(g)|.
\end{equation}
Moreover, for each subset $S \subset S_N$ and for every $g \in \I_N$ we have
\begin{equation}\label{WEgendom}
 |\phi^S(g)| \le
|c_\phi|\sum_{q \in \Q^\*}\left|
\prod_{p \mid N}W^S_p(a(q)g_{m(g_p),\ell(g_p),\nu(g_p)})
\prod_{p \nmid N}W_p(a(q))W_\infty(a(q)n(x)a(y)))\right|,
\end{equation}
where 
$W^S_p=W_p$ if $p \not \in S$, and $W^S_p$ is the normalized local newform attached to the  
contragradient $\tilde{\pi_p}$ if $p \in S$,
\end{lemma}

\begin{proof}
For each $p \mid N$, set $e_p=\lfloor \frac{n_p}2 \rfloor$. For convenience, also set $e'_p=\lceil \frac{n_p}2 \rceil$.
Then by Lemma~\ref{gendom}, we have 
$$\|\phi\|_{\infty} = \sup_{g \in \D_N \prod_{p \mid N}a(p^{e_p})} |\phi(g)| $$
We now prove that 
\begin{equation}\label{bound1}
\sup_{g \in \D_N \prod_{p \mid N}a(p^{e_p})} |\phi(g)| \le \max_{S}\sup_{g \in \I_N} | \phi^S(g)|.
\end{equation}
By Lemma~\ref{L1+DN} we have 
\begin{equation}\label{trivialbutpainful}
\sup_{g \in \D_N \prod_{p \mid N}a(p^{e_p})} |\phi(g)| \le
\sup_{\substack{
x \in \R, y \ge \frac{\sqrt 3}2\\
m(g_p)=-e_p \text{ if } \ell(g_p) \le \frac{n_p}2\\
m(g_p)=-2\ell(g_p)+e_p \text{ otherwise}
}}
 \left| \phi\left(n(x)a(y) \prod_{p \mid N} g_p\right)\right|. 
\end{equation}
By~(\ref{mirror}) we have
{ 
$$
g_{-2\ell_p+e_p,\ell_p,\nu_p}\begin{bmatrix}
      1 & \\
     &- \nu_p^{-2}
\end{bmatrix}
=
n(-p^{e_p-\ell_p}\nu^{-1})z(-p^{-\ell_p}\nu^{-1})
g_{-e'_p,n_p-\ell_p,-\nu_p}
\begin{bmatrix}
       & 1\\
    p^{n_p} & 
\end{bmatrix}.
$$}
Using this identity at each prime belonging to the set $S$ of primes $p$
satisfying $\ell(g_p) > \frac{n_p}2$ in the right hand side of~(\ref{trivialbutpainful})
we obtain by right-$K^{(1)}$ invariance of $\phi$
\begin{equation}\label{bound2}
\sup_{\substack{
x \in \R, y \ge \frac{\sqrt 3}2\\
m(g_p)=-e_p \text{ if } \ell(g_p) \le \frac{n_p}2\\
m(g_p)=-2\ell(g_p)+e_p \text{ otherwise}
}}
 \left| \phi\left(n(x)a(y) \prod_{p \mid N} g_p\right)\right|
 \le \max_{S \subset S_N}
 \sup_{\substack{
x \in \R, y \ge \frac{\sqrt 3}2\\
m(g_p)=-e_p \text{ if } p \in S\\
m(g_p)=-e'_p \text{ otherwise}\\
\ell(g_p) \le \frac{n_p}2
}}
 \left| \phi^S\left(n(x)a(y) \prod_{p \mid N} g_p\right)\right|.
\end{equation}
Combining~(\ref{trivialbutpainful}),~(\ref{bound2}) and the definition of $\I_N$, we obtain the bound~(\ref{bound1}).
From definition, it is clear that 
$$
\|\phi\|_{\infty} \ge \max_{S \subset S_N}\sup_{g \in \I_N} | \phi^S(g)|.
$$
so~(\ref{supofphi}) follows.

The second claim follows from the Whittaker expansion~(\ref{WE}). Observe that by Remark~\ref{conjnew}, 
$$\left|W_p\left(g_p \begin{bmatrix}
       & 1\\
    p^{n_p} & 
\end{bmatrix} \right)\right|=|\tilde{W_p}(g_p)|,$$ where $\tilde{W_p}$ is the normalized local newform attached to the contragradient $\tilde{\pi_p}$. The identity
$$a(q)n(x)=n(qx)a(q)$$ and the left invariance of the modulus of the local {  Whittaker} newforms by $NZ$ give~(\ref{WEgendom}).
\end{proof}

As we shall be interested in the support of the Whittaker expansion, we make now the following definition.
\begin{definition}
Keep notations as in Lemma~\ref{locgendom}.
For every $S \subset S_N$ and $g \in \I_N$ we define
$$Supp(g;S)=\left\{q \in \Q^\* :
\prod_{p \mid N}W^S_p(a(q)g_{m(g_p),\ell(g_p),\nu(g_p)})
\prod_{p \nmid N}W_p(a(q))W_\infty(a(q)n(x)a(y)) \neq 0\right\}.$$
\end{definition}


\begin{assumption}\label{assumption}
From now on we fix $g \in \I_N$ and $S \subset S_N$ (in the notations of \S~\ref{notation} and Definition~\ref{IN}),
and we define for each $p \mid N$, $\ell_p=\ell(g_p)$, $\epsilon_p=-m(g_p)$, $\epsilon'_p=n_p-\epsilon_p$, and $\nu_p=\nu(g_p)$.
We then define the following integers
\begin{equation*}
\begin{array}{cccc}
L=\prod_{p | N} p^{\ell_p}, 
&N_1=\prod_p p^{\epsilon_p}, 
&N_2=\prod_p p^{\epsilon'_p}, 
\end{array}
\end{equation*}
as well as the set of primes 
\begin{equation}\label{sets}
\begin{array}{cccc}
\Hm=\{ p \in \H : \ell_p < a_2(p)\},
&\He=\{ p \in \H : \ell_p= a_2(p)\},
&\Hp=\{ p \in \H : \ell_p> a_2(p)\},
\end{array}
\end{equation}
where $a_2(p)=n_p-c_p$ is the exponent of the conductor of the local character $\chi_2$
(note that in the case where $N=C$, we have $\Hm=\varnothing$ and $\He$ coincides with the set of primes
dividing $N$ and not dividing $L$). 
If $M=\prod_p p^{m_p}$ is any integer, we may use the notation 
\begin{equation*}
M^\star=\prod_{p \in \H_\star}p^{m_p}
\end{equation*}
for $\star \in \{+,-,=\}.$
\end{assumption}

\subsection{Sup norms: maximally ramified case}\label{maxram}
In this subsection, we are assuming $N=C$ and we prove Theorem~\ref{thmgen} in this special case,
as the proof becomes simpler.
We first determine the support of the ``Whittaker expansion"~(\ref{WEgendom}).

\begin{lemma}\label{support}
Recall {  Notation}~\ref{assumption}. 
There is a map 
\begin{align*}
\Psi(\He) &\to \{1, \cdots, L\} \\
s &\mapsto t_s
\end{align*}
such that 
$$Supp(g;S) \subseteq \left\{ 
\frac{s}{N_2L}(t_s+jL) : 
s \in \Psi(\He), 
j \in \Z
\text{ with }
t_s+jL
\text{ coprime to }
N
\right\}.$$
\end{lemma}

\begin{proof}
Let $q = \prod_p p^{q_p} \in \Q^\*$.
Assume $ q \in Supp(g;S)$
First, if $p \nmid N$ then we must have {  $q_p \ge 0$}.
So $\sgn(q)\prod_{p \nmid N}p^{q_p}$ is an integer. We shall see 
that it satisfies a certain congruence condition.
Consider now a prime $p \mid N$,
if $q=p^{q_p}u\in \Q^\*$ with $u\in \Z_p^\*$, we have
\begin{equation}\label{translateg}
a(q)g_{-\epsilon_p,\ell_p,\nu_p}=g_{q_p-\epsilon_p,\ell_p,\nu_p u^{-1}}
\begin{bmatrix}
   1 &  \\
     & u\\
\end{bmatrix}
=g_{q_p-\epsilon_p,\ell_p,\nu_p p^{q_p} q^{-1}}
\begin{bmatrix}
   1 &  \\
     & qp^{-q_p}\\
\end{bmatrix}.
\end{equation}

By Lemma~\ref{Wvaluesmax} (applied either to $\pi_p$ if $p \not \in S$ or to $\tilde{\pi_p}$ if $p \in S$), 
if $\ell_p=0$ then $q_p-\epsilon_p \ge -n_p$, so $q_p \ge -\epsilon'_p$.
It follows that $$s \doteq \prod_{\substack{p \mid N\\ p \nmid L}}p^{q_p+\epsilon_p'} \in \Psi(\He).$$
On the other hand, if $\ell_p > 0$ then  {  $q_p - \epsilon_p =-n_p - \ell_p$}, so 
$q_p =-\epsilon'_p - \ell_p$.
Now fix a prime $p_0 \mid L$ (so $\ell_{p_0} >0$), and write
\begin{align*}
\sgn(q)s\prod_{p \nmid N}p^{q_p} &=\sgn(q)
\prod_{p \nmid N}p^{q_p}\prod_{\substack{p \mid N\\ p \nmid L}}p^{q_p+\epsilon_p'}\\
&=\sgn(q)
\prod_{p \nmid N}p^{q_p}\prod_{\substack{p \mid N\\ p \nmid L}}p^{q_p+\epsilon_p'}
\prod_{\substack{p \mid L \\ p \neq p_0}} p^{q_p+\epsilon'_p+\ell_p}\\
&=\sgn(q)\prod_{p \neq p_0} p^{q_p}
\prod_{\substack{p \mid N\\ p \nmid L}}p^{\epsilon_p'}
\prod_{\substack{p \mid L\\p \neq p_0}}p^{\epsilon_p'+\ell_p}\\
&=
\left(p_0^{-q_{p_0}}q \right)
\left(
\prod_{\substack{p \mid N\\ p \nmid L}}p^{\epsilon_p'}
\prod_{\substack{p \mid L\\p \neq p_0}}p^{\epsilon_p'+\ell_p}\right).
\end{align*}
By Lemma~\ref{Wvaluesmax} and equality~(\ref{translateg}),
$p_0^{-q_{p_0}}q$ satisfies a certain congruence condition modulo $p^{\ell_{p_0}}\Z_{p_0}$.
In addition $\prod_{\substack{p \mid N\\ p \nmid L}}p^{\epsilon_p'}\prod_{\substack{p \mid L\\p \neq p_0}}p^{\epsilon_p'+\ell_p}$
 is clearly in $\Z_{p_0}^\*$.
So we just showed that the integer $\sgn(q)s\prod_{p \nmid N}p^{q_p}$ satisfies a certain congruence condition modulo $p_0^{\ell_{p_0}}$.
Applying the same reasoning with each prime dividing $L$, we obtain by the Chinese remainder theorem
a condition of the type
$$\sgn(q)s\prod_{p \nmid N}p^{q_p} \equiv r_0 \mod L.$$
 Since in addition $L$ and $s$ are coprime, we can write
\begin{equation}\label{coprimepart}
\sgn(q)\prod_{p \nmid N}p^{q_p}=t_s+jL
\end{equation} 
for some integer $t_s \equiv r_0 s^{-1} \mod L$, and $j$ ranging over $\Z$.
Finally, 
\begin{align*}
q=\sgn(q)\prod_{p \mid L}p^{-\epsilon_p'-\ell_p}\prod_{\substack{p \mid N\\ p \nmid L}}p^{q_p}\prod_{p \nmid N}p^{q_p}=\frac{s}{N_2L}(t_s+jL).
\end{align*}
\end{proof}

We now compute the size of each term in ``the Whittaker expansion"~(\ref{WEgendom}).
\begin{lemma}\label{size}
Keep notations from {  Notation}~\ref{assumption} and Lemma~\ref{locgendom}.
Let $q=\frac{s}{N_2L}(t_s+jL)$ as in Lemma~\ref{support}. 
Then we have 
$$\left|\prod_{p \mid N}W^S_p(a(q)g_{-\epsilon_p,\ell_p,\nu_p})\prod_{p \nmid N}W_p(a(q)) \right|=
L^{\frac12}s^{-\frac12}|t_s +jL|^{-\frac12}\left|\lambda_\pi(|t_s+jL|)\right|.$$
\end{lemma}
\begin{proof}
For $q$ of this form, using~(\ref{coprimepart}) and~(\ref{Lfunction}), we have 
\begin{align*}
\prod_{p \nmid N}W_p(a(q)) &= \prod_{p \nmid N}W_p(t_s+jL) \\
&=(|t_s +jL|)^{-\frac12}\lambda_\pi(|t_s+jL|),\\
\end{align*}
 and Lemma~\ref{Wvaluesmax} (observe that the contragradient representation $\tilde{\pi_p}$ satisfies the same hypothesis as $\pi_p$)
 together with equality~(\ref{translateg}) give
 \begin{align*}
 \left|\prod_{p \mid N}W^S_p(a(q)g_{-\epsilon_p,\ell_p,\nu_p}) \right| &=
\left|\prod_{p | N}W^S_p(g_{q_p-\epsilon_p,\ell_p,\nu_pp^{q_p}q^{-1}}) \right|\\
&= 
L^{\frac12}\prod_{ {\ell_p=0}}p^{-\frac{q_p-\epsilon_p+n_p}{2}} =L^{\frac12}s^{-\frac12}.
\end{align*}
\end{proof}

By Combining Lemmas~\ref{support} and~\ref{size} the ``Whittaker expansion"~(\ref{WEgendom}) is thus
bounded above by
$$c_\phi L^{\frac12}\sum_{s \in \Psi(\He)}s^{-\frac12}\sum_{j \in \Z}|t_s+jL|^{-\frac12+\delta+\epsilon}
\kappa\left(\frac{t_s+jL}{N_2L}sy\right).$$
Using estimate~(\ref{estimate}), we first evaluate the $j$-sum as follows:
\begin{align*}
\sum_{j \in \Z}|t_s+jL|^{-\frac12+\delta+\epsilon}
&\kappa\left(\frac{t_s+jL}{N_2L}sy\right) \\
&\ll 
\left(\frac{sy}{N_2L}\right)^{-\epsilon}\sum_{j \in \Z}|t_s+jL|^{-\frac12+\delta+\epsilon}
\exp \left((-2\pi+\epsilon)\frac{\left|t_s+jL\right|}{N_2L}sy\right)\\
&\ll \left(\frac{sy}{N_2L}\right)^{-\epsilon}\left(1+ \int_\R |tL|^{-\frac12+\delta+\epsilon}\exp \left((-2\pi+\epsilon)\frac{\left|t\right|}{N_2}sy\right)dt\right)\\
&\ll \left(\frac{N_2L}{sy}\right)^{\epsilon}\left(
1+\left(\frac{N_2}{Lsy}\right)^{\frac12}\left(\frac{N_2L}{sy}\right)^{\delta}
\right).
\end{align*}
Altogether, using Lemma~\ref{smoothpower} we get
$$|\phi(g)| \ll c_\phi\left(\frac{N_2L}{y}\right)^{\epsilon} \left(L^{\frac12}+\frac{N_2^{\frac12+\delta}L^\delta}{y}\right) \ll 
N^{\epsilon} \left(L^{\frac12}+N_2^{\frac12}N^{\delta}\right)$$
since $c_\phi \ll N^\epsilon$, $y \ge \frac{\sqrt 3}2$ and $N_2L \le N$. 
This establishes Theorem~\ref{thmgen} when $N=C$ because we have $L \le N^{\frac12}$ and $N_2\le \prod_{p|N}p^{\lceil \frac{n_p}2 \rceil}$.

\subsection{Sup norms: general ramification} 
Finally, let us address the necessary modifications when we do not make any assumption about the conductor of $\chi$. 
The analysis of the local Whittaker newform $W_p$ is similar, but with more cases to take into account, depending on which of the
sets~(\ref{highlyramifiedprimes}) the prime $p$ belongs.
In particular, it still holds that for all $p \in \H$ we have $\pi_p = \chi_1 \boxplus \chi_2$,
but the exponents $a_2(p)=n_p-c_p$ of the conductor of the local characters $\chi_2$ may not all equal zero.
We thus also get a Whittaker expansion supported on arithmetic progressions dictated by the primes 
{  at which the central character is	highly ramified}.
The rest of our argument differs from the maximally ramified case, as we rather use strong $L^2$-averages of the local
newforms, in the spirit of \cite{hybrid}, instead of the local bounds. 
Of course, in the maximally ramified case, these $L^2$-averages follow immediately from the computation of the support of the local
newform $W_p$ and the local bound, so the difference on the argument is mainly expository.
 
We first determine the support of the ``Whittaker expansion"~(\ref{WEgendom}) in
this more general case.
 \begin{lemma}\label{supporthigh}
Recall {  Notation}~\ref{assumption}. 
There is a map
\begin{align*}
\Psi(\He) \times \Psi(\L) & \to  \left\{ 1, \cdots, \frac{L^+C^+}{N^+} \right\}\\
(s,u) & \mapsto t_{su}
\end{align*}
such that 
\begin{align*}
Supp(g;S) \subseteq
\left\{
su\frac{N^+}{N_2L^+C^+}\left(t_{su}+j\frac{L^+C^+}{N^+}\right),
s \in \Psi(\He), u \in \Psi (\L), j \in \Z\right. \\ \left.\text{ with } t_{su}+j\frac{L^+C^+}{N^+} \text{ coprime to } N
\right\}.
\end{align*}
\end{lemma}

\begin{remark}
It is immediate by unravelling the definitions that $\frac{L^+C^+}{N^+}$ is an integer.
\end{remark}

\begin{proof}
The reasoning is quite similar to the proof of Lemma~\ref{support}, but we use~\cite{hybrid}*{Proposition~2.10} for the primes in $\L$
and Lemma~\ref{Wvalueshigh} instead of Lemma~\ref{Wvaluesmax} for those primes in {  $\H $.}
Fix $q = \prod_p p^{q_p} \in Supp(g;S)$.
As before, $\sgn(q) \prod_{p \nmid N} p^{q_p}$ is an integer and we shall see it satisfies some congruence condition.
If $p \in \He$ or $p \in \L$ then examination of either~Lemma~\ref{Wvalueshigh} or~\cite{hybrid}*{Proposition~2.10} gives  
$q_p \ge -\epsilon_p'$. So $$su \doteq \prod_{p \in \L \cup \He} p^{q_p + \epsilon'_p} \in \Psi(\L \cup \He).$$
In addition {  Lemma~\ref{Wvalueshigh} gives that} for $p \in \Hm$ we have {  $q_p=-\epsilon'_p$}, and 
for $p \in \Hp$ we have $q_p =\epsilon_p-\ell_p-a_1(p)$.
Fix $p_0 \in \Hp$ and write
{ 
\begin{align*}
\sgn(q)su\prod_{p \nmid N}p^{q_p} &=\sgn(q)
\prod_{p \nmid N}p^{q_p}\prod_{p \in \L \cup \He} p^{q_p + \epsilon'_p}\\
&=\sgn(q)
\prod_{p \nmid N}p^{q_p}\prod_{p \in \L \cup \He} p^{q_p + \epsilon'_p}
\prod_{p\in  \Hm}p^{q_p+\epsilon'_p}\prod_{\substack{p \in \Hp \\ p \neq p_0}}p^{q_p-(\epsilon_p-\ell_p-a_1(p))}\\
&=\sgn(q)\prod_{p \neq p_0} p^{q_p}
\prod_{p \in \L \cup \He \cup \Hm} p^{\epsilon'_p}
\prod_{\substack{p \in \Hp \\ p \neq p_0}}p^{\ell_p+a_1(p)-\epsilon_p}.
\end{align*}}
By Lemma~\ref{Wvalueshigh} $\sgn(q)\prod_{p \neq p_0} p^{q_p}$ satisfies a congruence condition modulo $p^{\ell_{p_0}-a_2(p_0)}\Z_{p_0}$.
Then using the Chinese remainder theorem we see that $\sgn(q)su\prod_{p \nmid N}p^{q_p}$ is an integer satisfying a congruence condition
modulo $\frac{L^+C^+}{N^+}$.
It follows that we can write
\begin{equation*}
\sgn(q)\prod_{p \nmid N}p^{q_p}=t_{su}+j\frac{L^+C^+}{N^+}.
\end{equation*}
Finally,
\begin{align*}
q&=\sgn(q)\prod_{p \nmid N}p^{q_p}\prod_{p \in \L \cup \He} p^{q_p}\prod_{p \in \Hm} p^{-\epsilon'_p}\prod_{p\in \Hp}p^{\epsilon_p-{  \ell_p}-a_1(p)}\\
&=\left(t_{su}+j\frac{L^+C^+}{N^+}\right)\frac{su}{\prod_{p \in \L} p^{\epsilon'_p}N_{2}^=}\frac{1}{N_2^-}\frac{N_1^+}{L^+C^+}.
\end{align*}
\end{proof}

If we were now to proceed following the exact same strategy as in the maximally ramified case, then we would get a worse 
estimate because of weaker local bounds for the local newform in the case 
$\ell_p=\frac{n_p}2$ (see~\cite{TAMS}*{Lemma~5.10}). Instead, we rely on $L^2$-averages of the local newvectors established by 
Saha~\cite{hybrid}.
To this end, we make first the following trivial lemma.
\begin{lemma}\label{trivialCS}
Suppose $(a_n)_{n \in \Z}$, $(b_n)_{n \in \Z}$ are two families of positive real numbers such that $\sum_{n\in \Z} a_nb_n$
converges absolutely\footnote{meaning the partial sums $\sum_{n\in J} |a_nb_n|$ indexed by finite sets $J \subset \Z$ are uniformly bounded.}, and $a_n$ is periodic with period $T$. Let $M$ be such that
$$\sum_{n=0}^{T-1} a_n^2 \le M.$$
Then we have 
$$\sum_{n \in \Z} a_nb_n \le M^{\frac12} \sum_{k \in \Z} 
\left(\sum_{j=0}^{T-1} b_{Tk+j}^2\right)^{\frac12}.$$
\end{lemma}
Next, we express the ``Whittaker expansion"~(\ref{WEgendom}) so as to be tackled by previous lemma.
\begin{lemma}\label{WEanbn}
Recall {  Notation}~\ref{assumption}. Then
$$|\phi^S(g)| \le |c_\phi|
\sum_{s \in \Psi(\He)}\sum_{u \in \Psi(\L)}\sum_{n \in \Z}a_nb_n,$$
where $a_n$ is periodic with period $L$ and satisfies
\begin{equation}\label{periodic}
\sum_{n=0}^{L-1} a_n^2 \ll N^\epsilon L(su)^{-\frac12}
\end{equation}
and $$b_n=
|n|^{-\frac12}
\left|
\lambda_\pi\left(n\right)\kappa\left(\frac{N^+suny}{N_2L^+C^+}\right)
\right|
\1_{n \equiv t_{su} \mod \frac{L^+C^+}{N^+}}.$$
\end{lemma}
\begin{proof}
The claim will follow from the ``Whittaker expansion"~(\ref{WEgendom}) 
\begin{equation*}
|\phi^S(g)| \le
|c_\phi|\sum_{q \in \Q^\*}\left|
\prod_{p \mid N}W^S_p(g_{q_p-\epsilon_p,\ell_p,\nu_pp^{q_p}q^{-1}})
\prod_{p \nmid N}W_p(a(q))W_\infty(a(q)n(x)a(y)))\right|,
\end{equation*}
together with~(\ref{lowestwv}),~(\ref{Lfunction}) and Lemma~\ref{supporthigh} once we have shown that the sequence defined by
 $$a_n=\left|\prod_{p \mid N}W^S_p\left(a\left(\frac{N^+sun}{N_2L^+C^+}\right)g_{-\epsilon_p,\ell_p,\nu_p}\right)\right|$$
 satisfies the desired properties. 
 For each $p \mid N$, let us {  distinguish cases depending on which of the sets defined
 	 in~(\ref{highlyramifiedprimes}) and~(\ref{sets}) contains $p$.}
 For all $v \in \Z_p^\*$ we have
$$
W^S_p\left(a\left(\frac{N^+suv}{N_2L^+C^+}\right)g_{-\epsilon_p,\ell_p,\nu_p}\right) =
\begin{cases}
W^S_p\left(a\left(v\right)g_{u_p-n_p,\ell_p,*}\right) \text{ if } p \in \L\\
W^S_p\left(a\left(v\right)g_{-n_p,\ell_p,*}\right) \text{ if } p \in \Hm\\
W^S_p\left(a\left(v\right)g_{s_p-n_p,\ell_p,*}\right) \text{ if } p \in \He\\
W^S_p\left(a\left(v\right)g_{-{  \ell_p}-a_1(p),\ell_p,*}\right) \text{ if } p \in \Hp,\\
\end{cases}
$$
where each $*$ is independent of $v$.
By~\cite{hybrid}*{Proposition~2.10}, we then get
$$
\int_{v \in \Z_p^\*}\left|W^S_p\left(a\left(\frac{N^+suv}{N_2L^+C^+}\right)g_{-\epsilon_p,\ell_p,\nu_p}\right)\right|^2d^\*v \ll
\begin{cases}
p^{-\frac{u_p}2} \text{ if } p \in \L\\
1 \text{ if } p \in \Hm\\
p^{-\frac{s_p}2} \text{ if } p \in \He\\
1 \text{ if } p \in \Hp.\\
\end{cases}
$$
Now by~\cite{hybrid}*{Remark~2.12}, for each $p \mid N$ 
 and each fixed $s \in \Psi(\He)$ and $u \in \Psi(\L)$, the map on $\Z_p^\*$ given by
$$v \mapsto \left|W^S_p\left(a\left(\frac{N^+suv}{N_2L^+C^+}\right)g_{-\epsilon_p,\ell_p,\nu_p}\right)\right|$$
is $U_p(\ell_p)$-invariant. Hence by the Chinese remainder theorem, these give rise to a map
on $\left( \Z / L \Z \right)^\*$ given by
$$(r \mod L) \mapsto \prod_{p \mid N} \left|W^S_p\left(a\left(\frac{N^+sur}{N_2L^+C^+}\right)g_{-\epsilon_p,\ell_p,\nu_p}\right)\right|,$$
and by Lemma~\ref{supporthigh}, if $a_n \neq 0$ then $n$ is coprime to $N$, thus the sum~(\ref{periodic}) is just
\begin{align*}
\sum_{r \in \left( \Z / L \Z \right)^\*}
 \prod_{p \mid N} \left|W^S_p\left(a\left(\frac{N^+sur}{N_2L^+C^+}\right)g_{-\epsilon_p,\ell_p,\nu_p}\right)\right|^2 
&=
\phi(L)\prod_{p \mid N} \int_{v \in \Z_p^\*}\left|W^S_p\left(a\left(\frac{N^+suv}{N_2L^+C^+}\right)g_{-\epsilon_p,\ell_p,\nu_p}\right)\right|^2d^\*v \\
&\ll N^\epsilon L(su)^{-\frac12},
\end{align*}
where $\phi$ is Euler's totient. 
\end{proof} 

By combining Lemmas~\ref{trivialCS} and~\ref{WEanbn} it follows
\begin{equation}\label{CSbound}
|\phi(g)| \ll N^\epsilon L^{\frac12} \sum_{s \in \Psi(\He)} s^{-\frac14}\sum_{u \in \Psi(\L)}u^{-\frac14}
\sum_{k \in \Z} S_k^\frac12,
\end{equation}
where 
\begin{equation}\label{defofSk}
S_k=\sum_{j=0}^{L-1} b_{Lk+j}^2,
\end{equation} and $b_n$ is defined in Lemma~\ref{WEanbn}.

\begin{lemma}\label{Sk}
For all $k \ge 1$ the sum~(\ref{defofSk}) satisfies 
$$S_k \ll \frac{N^+}{L^+C^+}L^{2\delta}\left(\frac{N}{suy}\right)^\epsilon k^{-1+2\delta+\epsilon}\exp\left(-\pi\frac{N^+suyL}{N_2L^+C^+}k\right),$$
and the same estimate holds for $k \le -2$ upon replacing $k$ with $-k-1$ in the right hand side.
Finally, 
$$S_0, S_{-1} \ll\left(\frac{N_2}{suy}\right)^\epsilon \left(1+\frac{N^+}{L^+C^+}\left(\frac{N_2L^+C^+}{N^+suy}\right)^{2\delta+\epsilon}\right).$$
\end{lemma}
\begin{proof}
For those intervals $[kL,(k+1)L]$ not containing zero we use estimate~(\ref{estimate}) then we
bound $S_k$ by the number of terms multiplied by the largest term. Since $\frac{L^+C^+}{N^+}$ divides $L$, the congruence condition
on $n=Lk+j$ modulo  $\frac{L^+C^+}{N^+}$ is equivalent to the same congruence condition on $j$. We thus get, for $k \ge 1$
\begin{align*}
S_k &\ll \frac{N^+}{L^+C^+}L^{2\delta}\left(\frac{N}{suy}\right)^\epsilon k^{-1+2\delta+\epsilon}\exp\left(-\pi\frac{N^+suyL}{N_2L^+C^+}k\right).
\end{align*}
For $k=0$ we have 
\begin{align*}
S_0 &\ll\left(\frac{N}{suy}\right)^\epsilon \left(1+\int_{0}^{\infty}
 \left(t_{su}+t\frac{L^+C^+}{N^+}\right)^{-1+2\delta+\epsilon} \exp\left(-\pi \frac{N^+suy}{N_2L^+C^+}\left(t_{su}+t\frac{L^+C^+}{N^+}\right)\right)dt\right)\\
&\ll \left(\frac{N}{suy}\right)^\epsilon \left(1+\frac{N^+}{L^+C^+}\left(\frac{N_2L^+C^+}{N^+suy}\right)^{2\delta+\epsilon}\right) .
\end{align*}
The analogous results for $k<0$ follow by changing $k$ to $-k-1$ and $t_{s,u}$ to 
$\frac{L^+C^+}{N^+}-t_{s,u}$.
\end{proof}

By a similar argument as in \S~\ref{maxram}, Lemma~\ref{Sk} implies 
$$\sum_{k \in \Z}S_k^{\frac12} \ll \left(\frac{N}{suy}\right)^\epsilon 
\left(1+\left(\frac{N^+}{L^+C^+}\right)^{\frac12}\left(\frac{N_2L^+C^+}{N^+suy}\right)^{\delta+\epsilon} + 
\left(\frac{N_2}{Lsuy}\right)^{\frac12}\left(\frac{N_2L^+C^+}{N^+suy}\right)^{\delta+\epsilon}\right) $$ 
 {  Substituting} this into~(\ref{CSbound}) and using Lemma~\ref{smoothpower} we
obtain
$$
|\phi(g)| \ll N^{\delta+\epsilon} \left(
L^{\frac12}
+N_2^{\frac12}\right).
$$
{  Lemma~\ref{locgendom} together with the results from Section~3.4 and~3.5 finishes the proof of Theorem~\ref{thmgen}}.

\end{document}